\documentclass[a4paper]{amsart}

\usepackage{amssymb}

\usepackage{mathdots}
\newtheorem{thm}{Theorem}[section]
\newtheorem{cor}[thm]{Corollary}
\newtheorem{lem}[thm]{Lemma}
\newtheorem{prop}[thm]{Proposition}

\theoremstyle{remark}

\theoremstyle{definition}

\begin{document}
\pagestyle{myheadings}
\address{Department of Mathematics, Technion - Israel Institute of Technology, 32000, Haifa, Israel.}
\email{max@tx.technion.ac.il}
\title{A Distributional Treatment of Relative Mirabolic Multiplicity One}
\author{Maxim Gurevich}
\date{\today}

\begin{abstract}
We study the role of the mirabolic subgroup $P$ of $G=\mathbf{GL}_n(F)$ ($F$ a $p$-adic field) in smooth irreducible representations of $G$ that possess a non-zero invariant functional relative to a subgroup of the form $H_{k} = \mathbf{GL}_k(F)\times \mathbf{GL}_{n-k}(F)$. We show that if a non-zero $H_1$-invariant functional exists on a representation, then every $P\cap H_1$-invariant functional must equal to a scalar multiple of it. When $k>1$, we give a reduction of the same problem to a question about invariant distributions on the nilpotent cone of the tangent space of the symmetric space $G/H_k$. Some new distributional methods, which are suitable for a setting of non-reductive groups, are developed.

\end{abstract}

\maketitle

\section{Introduction}
We study the nature of smooth irreducible complex representations of $G_n = \mathbf{GL}_n(F)$ over a non-archimedean field $F$ of characteristic $0$, that are distinguished by subgroups of the form $H_{k,n}=\mathbf{GL}_{k}(F)\times\mathbf{GL}_{n-k}(F)$ (the maximal standard Levi subgroups). Given a smooth irreducible representation $\pi$ of $G_n$ on a space $V$, it is said that $\pi$ is distinguished by a subgroup $H<G_n$ if there is a non-zero linear $H$-invariant functional on $V$. Equivalently, $\pi$ appears as a sub-representation of $C^\infty(G_n/H)$. Apart from standalone interest in the representation theory of $p$-adic groups, distinction comes into play as a local accessory when dealing with periods of an automorphic representation.
\par For such $(\pi,V)$, the dual $V^\ast$ can be embedded as a $G_n$-space into the space of distributions $\mathcal{D}$ on $G_n/H$. This observation allows one to push geometric results on the $p$-adic space $\mathcal{D}$ into representation theory. Among the advantages of such an approach, which has long been known for its potency, is the absence of reliance on a classification of the representations of $G_n$.
\par For one instance, the mentioned motivation from automorphic forms raises the question of multiplicity one, that is, can the $H$-invariants of $V^\ast$ be larger than one-dimensional? In the case of $H_{k,n}$ it was answered negatively by Jacquet and Rallis \cite{jac-ral}, using novel distributional methods. Namely, they applied a transfer of distributions from the symmetric space $G_n/H_{k,n}$ to its linear tangent space. The method was rephrased by Aizenbud and Gourevitch in \cite{aiz-gur1} as a part of a more elaborate mechanism called the Harish-Chandra descent. In this account, we would like to explore further these ideas to address another distributional problem with an implication for representation theory.
\par We focus on the mirabolic subgroup $\mathbf{P}_n(F)<G_n$, that is, the stabilizer subgroup in $G_n$ of the vector $(0,\ldots,0,1)$ in its natural action on the row space $F^n$. Our main result regards the collection of $H_{1,n}$-distinguished representations of $G_n$. 
\begin{thm}\label{start-thm}
Every $\mathbf{P}_n(F)\cap (\mathbf{GL}_1(F)\times \mathbf{GL}_{n-1}(F))$-invariant linear functional on a $\mathbf{GL}_1(F)\times \mathbf{GL}_{n-1}(F)$-distinguished irreducible smooth representation of $\mathbf{GL}_{n}(F)$, is also $\mathbf{GL}_1(F)\times \mathbf{GL}_{n-1}(F)$-invariant.
\end{thm}

In particular, this implies multiplicity one for $\mathbf{P}_n(F)\cap H_{1,n}$-invariant functionals on $H_{1,n}$-distinguished representations.
\par Let us mention, that from the results of \cite{venke} some explicit constructions of $H_{1,n}$-distinguished representations can be produced, for which Theorem \ref{start-thm} may be applied. For example, if $\sigma$ is a smooth irreducible representation of $G_2$ with trivial central character and $\pi = \sigma \times 1_{n-2}$ (in the sense of the Bernstein-Zelevinsky classification, where $1_{n-2}$ is the trivial representation of $G_{n-2}$) is irreducible, then $\pi$ is $H_{1,n}$-distinguished. Thus, Theorem \ref{start-thm} concerns what can be roughly described as embeddings of the smooth spectrum of $\mathbf{PGL}_2(F)$ into that of $G_n$. On the other hand, it follows from the results of \cite{matr-linear}, that if $n>3$, any $H_{1,n}$-distinguished $\pi$ must be non-generic.
\par Theorem \ref{start-thm} comes as a corollary of our study into the geometric question about the difference between $P_n(F)\cap H_{k,n}$-invariance and $H_{k,n}$-invariance of distributions on $G_n/X_{k,n}$. We show (Theorem \ref{final-form}) that for all $1\leq k<n-2$ the question can be reduced to that of invariant distributions on a certain cone in a linear $F$-space. Thus, we complete the following reduction.

\begin{lem}\label{start-lem}
Let $1\leq k\leq n-2$ be integers. \\
For $m\leq l$, let $\mathcal{N}_{m,l}$ be the space of nilpotent matrices of the form 
$$\left(\begin{array}{cc} 0 & B \\ C & 0 \end{array} \right)\;\quad\; B\in \mathbf{M}_{m,l-m}(F)\quad C\in \mathbf{M}_{l-m,m}(F).$$
Suppose that for all $1\leq m\leq k$ and $m\leq l \leq n$,  every $\mathbf{P}_l(F)\cap (\mathbf{GL}_m(F)\times \mathbf{GL}_{l-m}(F))$-conjugation invariant distribution on $\mathcal{N}_{m,l}$ is also $\mathbf{GL}_m(F)\times \mathbf{GL}_{l-m}(F)$-invariant. Then, every $\mathbf{P}_n(F)\cap (\mathbf{GL}_k(F)\times \mathbf{GL}_{n-k}(F))$-invariant linear functional on a $\mathbf{GL}_k(F)\times \mathbf{GL}_{n-k}(F)$-distinguished irreducible smooth representation of $\mathbf{GL}_{n}(F)$, is also $\mathbf{GL}_k(F)\times \mathbf{GL}_{n-k}(F)$-invariant.
\end{lem}
The role of the mirabolic subgroup is known to be ubiquitous in the representation theory of $G_n$. For example, the subgroup is seen in the Gelfand-Kazhdan theory of derivatives, which served as a foundation for the Zelevinsky classification. In \cite{bern}, Bernstein showed a distributional result regarding this subgroup: Every $\mathbf{P}_n(F)$-conjugation-invariant distribution on the matrix space $\mathbf{M}_{n,n}(F)$ is also $G_n$-conjugation-invariant. Using this result, he showed that integration over the mirabolic group in the Whittaker model of generic representations defines a canonical inner product between a representation and its contragredient. Our current study can be seen as a follow up to these findings. We would like to check to what extent the role of the mirabolic group is preserved in the relative setting. 
\par In the context of distinction, Bernstein's distributional result was applied in \cite{off} to prove a similar statement to Theorem \ref{start-thm} for $\mathbf{GL}_n(L)$-distinguished representations, where $F/L$ is a quadratic field extension. In our case, i.e.~$H_{k,n}$-distinction, we show that the geometry of the space requires a similar approach to be supplemented with the assumption of Lemma \ref{start-lem}, that is, $\mathcal{D}(\mathcal{N}_{k,n})^{\mathbf{P}_n(F)\cap H_{k,n}} = \mathcal{D}(\mathcal{N}_{k,n})^{H_{k,n}}$.  The relative analogy to Bernstein's result facilitates the conjecture that the assumption indeed holds (for all $1\leq k\leq n-2$). We were successful in proving it for the case $k=1$ by techniques of prolongation of invariant distributions.
\par Let us expound further on this. In our proof we are able to prolong $H_{1,n}$-invariant distributions on $\mathcal{N}_{1,n}$ from one open orbit to its closure. Yet, in \cite[Section 4]{rr}, an example of a similar case was shown in which a so-called Ranga-Rao type theorem does \textit{not} hold. Namely, if $G/H$ is a $p$-adic symmetric space, the $H$-invariant distributions on the nilpotent orbits of $Lie(G)\ominus Lie(H)$ may not possess such prolongation properties. Indeed, the validity of such a property in our case for $k>1$ remains one of the obstacles when dealing with the general linear nilpotent problem (the assumption of Lemma \ref{start-lem}). 
\par Section 2 sets the main distributional tools needed for our analysis. We recall the Frobenius descent (Proposition \ref{frob}) which allows for the most direct transfer of a distribution into a smaller space. Since the mirabolic group is not reductive, we are unable to apply ``traditional" distributional techniques directly. For that reason we introduce a new refinement (Proposition \ref{abstract-lem}) of the descent techniques, which eases the treatment of invariance under general $p$-adic groups. We also sketch the treatment of Luna's Slice Theorem in the setting of symmetric spaces which was developed in \cite{aiz-gur1}.
\par In addition, we give a formulation (Proposition \ref{dense-frob}) of what we call the dense Frobenius descent. Building upon ideas from \cite{bern}, it will serve us as a tool for reducing problems of invariant prolongation of distributions into same problems on smaller spaces.
\par Section 3 studies the geometry of closed $H_{k,n}$-orbits on the space $G_n/H_{k,n}$ and their decomposition to $\mathbf{P}_n(F)\cap H_{k,n}$-orbits. Using the mentioned Harish-Chandra descent techniques we are able to reduce a question on the distribution spaces of $G_n/H_{k,n}$ to that of distributions on its tangent space (Theorem \ref{main-reduction}). Furthermore, we show that the ``heart" of the problem lies in distributions on the nilpotent cones $\mathcal{N}_{k,n}$.
\par In Section 4, using the dense Frobenius descent, we construct the full space of $H_{1,n}$-invariant distributions on $\mathcal{N}_{1,n}$. It is shown that these are also invariant under the smaller $\mathbf{P}_n(F)\cap H_{1,n}$ (Theorem \ref{k1-thm}). That will allow the final deduction of Theorem \ref{start-thm} from Lemma \ref{start-lem}.
\par For completeness, we give in Section 5 the details of the embedding of the linear forms on a $H_{k,n}$-distinguished irreducible representation of $G_n$ into the space of distributions on $G_n/H_{k,n}$. That description is rather classical, and completes our transfer of distributional results into representation theory.
\par The results reported in this account are part of my Ph.D. research. I would like to thank Omer Offen for suggesting me this problem and providing guidance. Special thanks to Dima Gourevitch and Rami Aizenbud for useful discussions and suggestions.

\section{Tools and preliminaries}
\subsection{Notation}
Let $F$ be a non-archimedean local field of characteristic $0$. We will write $\mathbf{V}(F)$ for the $F$-points of an algebraic variety $\mathbf{V}$ defined over $F$. Specifically, we fix the notation $G_n = \mathbf{GL}_n(F)$ and also denote the naturally embedded subgroups $H_{k,n} = G_k\times G_{n-k}<G_n$, for $0\leq k< n$.
\par Let $\theta_{k,n}: \mathbf{GL}_n\to \mathbf{GL}_n$ be the involutive automorphism defined by 
$$\theta_{k,n}(g) = \epsilon_{k,n} g\epsilon_{k,n}\qquad \epsilon_{k,n}=\left(\begin{array}{ll} I_k & 0 \\ 0 & -I_{n-k}\end{array}\right).$$
Then, the fixed point subgroup of $\theta_{k,n}$ is exactly $\mathbf{GL}_k\times \mathbf{GL}_{n-k}$, which makes the quotient $\mathbf{GL}_n/(\mathbf{GL}_k\times \mathbf{GL}_{n-k})$ a \textit{symmetric space}. For our current needs this will allow the following construction: Consider the left action of $G_n$ on itself by $\theta_{k,n}$-twisted conjugation: $g\cdot x := gx\theta_{k,n}(g)^{-1}$. The stabilizer of the identity element will be the fixed point subgroup of $\theta_{k,n}$ inside $G_{n}$, i.e.~$H_{k,n}$. Thus, the orbit of the identity relative to this action
$$X_{k,n} = \left\{ g\theta_{k,n}(g)^{-1}\;:\; g\in G_n\right\}\subset G_n$$
is identified with the quotient $G_n/H_{k,n}$. We also denote the symmetrization map $\rho_{k,n}(g) = g\theta_{k,n}(g)^{-1} = g\cdot I_n$ from $G_n$ to $X_{k,n}$.
Notice that the action of $H_{k,n}$ on $X_{k,n}$ is given by usual conjugation, and therefore the stabilizer group inside $H_{k,n}$ of a point $x\in X_{k,n}$ is the centralizer of $x$ which we will denote by $H^x_{k,n}$. A similar notation for a centralizer will be adopted for other groups as well. 
\par Observe that on the Lie algebra of $\mathbf{GL}_n$ there is a linear involution $d\theta_{k,n} = Ad(\epsilon_{k,n})$. After identifying the algebra with the matrix space $\mathbf{M}_{n,n}$ it becomes useful to define the linear version of the symmetric space
$$\mathbf{L}_{k,n} = \left\{ A\in \mathbf{M}_{n,n}\;:\; Ad(\epsilon_{k,n})(A) = -A\right\} =$$ $$= \left\{\left(\begin{array}{cc} 0 & B \\ C & 0 \end{array} \right)\;:\; B\in \mathbf{M}_{k,n-k}, C\in \mathbf{M}_{n-k,k}\right\}.$$

\par Consider the natural right action of $\mathbf{GL}_n$ on the row vector space $\mathbb{A}^n$. Let $\mathbf{P}_n< \mathbf{GL}_n$ be the stabilizer group of the vector $(0,\ldots,0,1)$. This is the standard mirabolic group which clearly consists of matrices whose bottom row is $(0 \cdots 0 1)$. We also denote the intersection $P_{k,n} = H_{k,n} \cap \mathbf{P}_n(F)$.

\par For a general locally compact totally disconnected (lctd) space $X$ (such as the $F$-points of an algebraic group), we write $\mathcal{S}(X)$ for the space of locally constant compactly supported complex functions on $X$. The space $\mathcal{D}(X)$ of distributions on $X$ is defined to be the dual space of $\mathcal{S}(X)$. Recall that given an open subset $\Omega$ of such a space $X$, $\Omega$ and $X\setminus \Omega$ become lctd spaces themselves, and we have a short exact sequence
$$0\longrightarrow \mathcal{D}(X\setminus \Omega)\longrightarrow \mathcal{D}(X)\longrightarrow \mathcal{D}(\Omega)\longrightarrow 0$$
Thus, for $T\in \mathcal{D}(X)$ we can talk about its restriction $T|_{\Omega}$, or say it is \textit{supported} in $X\setminus \Omega$. The support $supp\, T$ of a distribution $T$ is the complement of the union of all open sets $\Omega$ for which $T|_{\Omega}=0$.
\par A continuous left action of a lctd group $G$ on a lctd space $X$ induces a left linear action of $G$ on $\mathcal{S}(X)$ by $(g\cdot f)(x) = f(g^{-1}\cdot x)$ ($g\in G, f\in \mathcal{S}(X), x\in X$). Hence, $G$ acts on $\mathcal{D}(X)$ by the dual action.
\par In general, for a group $G$ that acts linearly on a complex space $V$ and a group character $\chi:G\to \mathbb{C}$, we will use the notation $V^{G,\chi} = \{v\in V\;:\; g\cdot v = \chi(g)v\;\forall g\in G\}$ for the space of semi-invariants. As usual $V^G = V^{G,1}$. Given an involution $\sigma$ on any vector space $X$, we will denote by $X^{\widehat{\sigma}}= \{x\in X\;:\; \sigma(x)=-x\}$ the space of anti-invariants of $X$. E.g.~$\mathbf{L}_{k,n}(K) = \mathbf{M}_{n,n}(K)^{\widehat{Ad(\epsilon_{k,n})}}$ for any field $K$.
\par We will denote by $\Delta_G$ the modular character attached to a lctd group $G$, and given a closed subgroup $H<G$ we define the relative modular character $\Delta_{G/H} = \Delta_G|_H \Delta_H^{-1}$ of $H$.
\par Recall that the modular character of $\mathbf{P}_n(F)$ is given by $|\det^{-1}|$, where $|\cdot|$ is the standard absolute value of $F$. Note also, that the same formula remains true for the group $\mathbf{P}_n(E)$, where $F<E$ is a finite field extension, if the determinant of a matrix is taken as an $F$-linear operator.

\subsection{Frobenius Descent}
A major focal point of this work is the analysis of certain spaces of invariant distributions. One valuable tool available to deal with these spaces is the so-called Frobenius descent. It is a well-known rephrasing of the Frobenius reciprocity of group representations. After recalling the original formulation of the descent technique, we develop some adaptations of the tool to our needs.
\begin{prop}\label{frob}\cite[Lemma 1.5]{bern}
Suppose $G$ is a lctd group acting continuously on lctd spaces $\mathcal{X}$ and $\mathcal{O}$, with a $G$-equivariant continuous map $p:\mathcal{X}\to \mathcal{O}$ between them. Suppose furthermore, that the action on $\mathcal{O}$ is transitive. For a chosen point $x\in \mathcal{O}$ let $H<G$ be the stabilizer group of $x$. Then, for any character $\chi:G\to \mathbb{C}$ there is an isomorphism of distribution spaces
$$\mathcal{D}(p^{-1}(x))^{H,\chi \Delta_{G/H}^{-1}} \cong \mathcal{D}(\mathcal{X})^{G,\chi}.$$ 
If a distribution $T'\in\mathcal{D}(p^{-1}(x))^{H,\chi \Delta_{G/H}^{-1}}$ corresponds to $T\in  \mathcal{D}(\mathcal{X})^{G,\chi}$, we have the equality $supp\, T = G\cdot supp\,T'$.
\end{prop}
We further mention that if $\Delta_{G/H}$ is trivial (with the assumptions of the above proposition), then $\mathcal{O}$ has a non-zero $G$-invariant distribution $\mu$ on it. In that case the Frobenius descent map $\Phi_\mu: \mathcal{D}(p^{-1}(x))^{H} \to \mathcal{D}(\mathcal{X})^{G}$ can be given by the formula
$$(\Phi_\mu(T))(f)= \int_\mathcal{O} T((g_z\cdot f)|_{p^{-1}(x)})\,d\mu(z)\qquad f\in\mathcal{S}(\mathcal{X})$$
where $g_z\in G$ are fixed elements such that $g_z\cdot z = x$. 

\par Given a group $G$ acting on $X$ and its subgroup $P$, we want to formulate a criterion for the equality of spaces $\mathcal{D}(X)^G$ and $ \mathcal{D}(X)^P$ based on the Frobenius descent. 

\begin{prop}\label{abstract-lem}
Suppose $G$ is an lctd group, and $P<G$ is a closed subgroup. Suppose $\mathcal{X}$ and $\mathcal{O}$ are lctd spaces on which $G$ acts, and $p:\mathcal{X}\to \mathcal{O}$ is an $G$-equivariant surjection. Assume: \\
(i) The action on $\mathcal{O}$ is $G$-transitive. \\
(ii) There is a non-zero $G$-invariant distribution on $\mathcal{O}$.\\
(iii) There are finitely many $P$-orbits in $\mathcal{O}$.\\ 
If there exists an element $x\in \mathcal{O}$ and an open $P$-orbit $\mathcal{U}\subset \mathcal{O}$ containing $x$ such that:\\
(iv) $\mathcal{D}(p^{-1}(x))^{Stab(P,x)} = \mathcal{D}(p^{-1}(x))^{Stab(G,x)}$.\\
(v) For every $y\in \mathcal{O}\setminus \mathcal{U}$, $\mathcal{D}(p^{-1}(y))^{Stab(P,y),\, \Delta_{P/Stab(P,y)}^{-1}}=\{0\}$.\\
then 
$$\mathcal{D}(\mathcal{X})^{P} = \mathcal{D}(\mathcal{X})^G.$$
\end{prop}

\begin{proof}
Let $T\in \mathcal{D}(\mathcal{X})^{P}$ be a distribution. If $0\neq \mu$ is an $G$-invariant distribution on $\mathcal{O}$, then $0\neq \mu|_{\mathcal{U}}$ is $P$-invariant, and $\Delta_{P/Stab(P,x)}=1$ follows. 
Thus, we can use Proposition \ref{frob} on the mapping $p: p^{-1}(\mathcal{U})\to \mathcal{U}$ to produce $\xi\in \mathcal{D}(p^{-1}(x))^{Stab(P,x)}$ corresponding to $T|_{p^{-1}(\mathcal{U})}$. 
From condition (iv) we know that $\xi$ is also $Stab(G,x)$-invariant, hence, corresponds by another Frobenius descent to an $G$-invariant distribution $S$ on $\mathcal{X}$. Yet, looking on both descents through the formula given above we see that $S|_{p^{-1}(\mathcal{U})} = T|_{p^{-1}(\mathcal{U})}$, which means $T-S$ has its support in $p^{-1}(\mathcal{O}\setminus \mathcal{U})$. Therefore, it is enough to prove that $\mathcal{D}(p^{-1}(\mathcal{O}\setminus \mathcal{U}))^P$ is trivial to show that $T=S$ and finish.\\
This is achieved by induction on the number of $P$-orbits in $\mathcal{O}\setminus \mathcal{U}$. Since this is a finite number, there is an open $P$-orbit $\mathcal{U}_1\subset \mathcal{O}\setminus \mathcal{U}$. Choosing $y\in \mathcal{U}_1$ and combining condition (v) with a Frobenius descent, we see that there are no non-zero $P$-invariant distributions on $p^{-1}(\mathcal{U}_1)$. Thus $\mathcal{D}(p^{-1}(\mathcal{O}\setminus \mathcal{U}))^P =\mathcal{D}(p^{-1}(\mathcal{O}\setminus (\mathcal{U}\cup \mathcal{U}_1)))^P$, and so on.
\end{proof}

\par Another problem we can tackle with descent techniques is the prolongation of invariant distributions. Given a distribution on a locally closed subset $Y$ of a larger space we would like to claim that it is the restriction of a distribution on $\overline{Y}$ with the same invariance properties. The following proposition allows us under suitable conditions to reduce such a question to a smaller space. This argument has appeared implicitly in \cite[4.3]{bern}.
\begin{prop}[``Dense Frobenius descent"]\label{dense-frob}
Let $G$ be a lctd group that acts on a lctd space $\mathcal{X}$. Let $\mathcal{K}$ be a compact totally disconnected space on which $G$ acts transitively.
Suppose $\mathcal{Y}\subset \mathcal{X}$ is an $G$-invariant subset equipped with an $G$-equivariant surjective continuous map $i:\mathcal{Y}\to \mathcal{K}$.\\
For a point $y\in \mathcal{K}$, denote $\mathcal{Y}_y= i^{-1}(y)$ and $B= Stab(y)< G$. If there exists a distribution $T\in\mathcal{D}(\mathcal{X})^{B,\Delta_{G/B}^{-1}}$ with $supp\, T = \overline{\mathcal{Y}_y}$, then there also exists a distribution $T'\in \mathcal{D}(\mathcal{X})^G$ with $supp\, T' = \overline{\mathcal{Y}}$.
\end{prop}
\begin{proof}
Consider the product space $\mathcal{X}\times \mathcal{K}$ with its $G$-action. Let $Q\subset \mathcal{X}\times \mathcal{K}$ be the graph of the map $i$. We apply Proposition \ref{frob} relative to the projection $p_2:\mathcal{X}\times \mathcal{K}\to \mathcal{K}$ on the right component. This gives the correspondence $\mathcal{D}(p_2^{-1}(y))^{B,\Delta_{G/B}^{-1}}\cong \mathcal{D}(\mathcal{X}\times \mathcal{K})^G$. Identifying $p_2^{-1}(y)$ with $\mathcal{X}$, we see that the existence of $T$ as described in the assumption supplies the existence of a $G$-invariant distribution $T''$ on $\mathcal{X}\times \mathcal{K}$ whose support is $\overline{G\cdot (\overline{\mathcal{Y}}_y\times \{y\})}= \overline{Q}$.\\
Consider the projection $p_1:\mathcal{X}\times \mathcal{K}\to \mathcal{X}$ on the left component. Since $\mathcal{K}$ is compact, if $f$ is in $\mathcal{S}(\mathcal{X})$, then $f\circ p_1\in\mathcal{S}(\mathcal{X}\times \mathcal{K})$. Hence, we can push distributions with $p_{1\ast}:\mathcal{D}(\mathcal{X}\times \mathcal{K})\to \mathcal{D}(\mathcal{X})$ defined by $p_{1\ast}(T)(f) = T(f\circ p_1)$. \\
We claim that the support of $T':= p_{1\ast}(T'')$ is exactly $\overline{\mathcal{Y}}$. Indeed, for $x\in \mathcal{Y}$, we have $(x,i(x))\in supp\, T''$. Therefore there exists an open compact neighborhood $U\times V$ of $(x,i(x))$ with $T''(\chi_{U\times V})\neq 0$ (with $\chi$ denoting the characteristic function of a set). From continuity of $i$ we can assume that $i(U\cap \mathcal{Y})\subset V$, hence, $U\times( \mathcal{K}\setminus V)$ lies outside the support of $T''$. So, $T'(\chi_U) = T''(\chi_{U\times V})\neq 0$.
\end{proof}

\subsection{Luna's Slice Theorem and applications}
In \cite{aiz-gur1}, the authors developed a mechanism for the transfer of information on spaces of invariant distributions on $F$-varieties to invariant distributions on $F$-linear spaces. We will now review the required tools from the mentioned reference.
\par Suppose a linear reductive $F$-group $\mathbf{G}$ acts on an affine $F$-variety $\mathbf{V}$. We will say that a subset $U\subset \mathbf{V}(F)$ is \textit{saturated} if there is a $\mathbf{G}(F)$-invariant continuous map $\pi:\mathbf{V}(F)\to F^m$ (such map always exists) and an open set $V\subset F^m$ such that $U = \pi^{-1}(V)$. In particular, saturated subsets are open and $\mathbf{G}(F)$-invariant.
\par In the same setting, suppose $x\in \mathbf{V}(F)$ is such that its $\mathbf{G}(F)$-orbit $\mathcal{O}$ is closed (equivalently, its $\mathbf{G}$-orbit is Zariski closed). The tangent spaces $T_x(\mathbf{G}\cdot x)\subset T_x\mathbf{V}$ are well-defined and are equipped with the action of the stabilizer group $\mathbf{G}^x$ of $x$ (by differentiation). So, $\mathbf{G}^x$ acts on their quotient which is the normal space $\mathbf{N}_x(\mathbf{G}\cdot x)$ whose $F$-points may be identified with the normal space $N_x\mathcal{O}$ of $F$-analytic manifolds.
\begin{thm}\label{luna}\textbf{Luna's Slice Theorem }(see e.g.~\cite[2.3.17]{aiz-gur1}) \\ 
Let $\mathbf{G}$ be a linear reductive $F$-group acting on an affine $F$-variety $\mathbf{V}$, and $x\in \mathbf{V}(F)$ such that its $\mathbf{G}(F)$-orbit $\mathcal{O}$ is closed. Then there exists an open $\mathbf{G}(F)$-invariant neighborhood $\mathcal{O}\subset U\subset \mathbf{V}(F)$, a $\mathbf{G}(F)$-equivariant continuous retract $p: U\to\mathcal{O}$, and a $\mathbf{G}^x(F)$-equivariant $F$-analytic embedding $\psi: p^{-1}(x)\to N_x\mathcal{O}$ whose image is saturated and $\psi(x)=0$.
\end{thm}

Let us remark that the slices of the variety obtained in the above procedure exhaust the whole variety because of the following fact.
\begin{prop}\label{covered}\cite[2.3.7]{aiz-gur1}
Suppose a linear $F$-group $\mathbf{G}$ acts on an affine $F$-variety $\mathbf{V}$. If $U$ is an open subset of $\mathbf{V}(F)$ which contains all closed $\mathbf{G}(F)$-orbits on $\mathbf{V}(F)$, then $U=\mathbf{V}(F)$.
\end{prop}

We will need to apply this linearization technique on the action of $H_{k,n}$ on $X_{k,n}$. This is possible because we may view $X_{k,n}\cong G_n/H_{k,n}$ as embedded inside the $F$-variety $(\mathbf{GL}_n/\mathbf{GL}_k\times \mathbf{GL}_{n-k})(F)$. As we have mentioned, for $x\in X_{k,n}$ the stabilizer group of the $H_{k,n}$-action is $H_{k,n}^x$.
\begin{prop}\label{closed}
If $x\in X_{k,n}$ has a closed $H_{k,n}$-orbit, then $x$ is a semisimple matrix, and the action of $H^x_{k,n}$ on the normal space $N_x\mathcal{O}$ is isomorphic to its action on $\mathbf{L}_{k,n}^{Ad(x)}(F)$ by conjugation.
\end{prop}
\begin{proof}
This is \cite[Proposition 7.2.1]{aiz-gur1}, after noting that the normal space in $\mathbf{GL}_n$ relative to the action of $\mathbf{GL}_k\times \mathbf{GL}_{n-k}$ on both sides is isomorphic to the normal space of $\mathbf{GL}_n/\mathbf{GL}_k\times \mathbf{GL}_{n-k}$ relative to the one sided action. Moreover, the closed orbits of both actions are in correspondence.
\end{proof}

\section{Reduction to a linear problem}
As will be explained in Section 5, the answer to the representation theoretic problems with which we are dealing can be deduced from the conjectural equalities $\mathcal{D}(X_{k,n})^{P_{k,n}}= \mathcal{D}(X_{k,n})^{H_{k,n}}$ of distribution spaces. We would like to reduce the problem of proving such an equality into an equality of distribution spaces on a given cone of a linear space.
\par Let us introduce some terminology. When a group $G$ acts on a space $X$, we would like to call the pair $(G,X)$ an \textit{action space}. Given a set $\{(G^i,X^i)\}_{i=1}^t$ of action spaces, we will say that an action space $(G,X)$ is their \textit{product}, if there exists a group isomorphism $G\cong G^1\times\cdots\times G^t$ and a bijection $X \to X^1\times\cdots\times X^t$ which intertwines the $G$-action relative to the fixed isomorphism.
\par Suppose an action space $(G,X)$ admits a decomposition into a product of $\{(G^i,X^i)\}_{i=1}^t$. We call a subgroup $H< G$ \textit{admissible} to this decomposition if the fixed isomorphism $G\cong G^1\times\cdots\times G^t$ sends $H$ to a product of subgroups $H^i<G^i$. In this case $(H,X)$ becomes the product of $\{(H^i, X^i)\}_{i=1}^t$.

\subsection{Geometry of closed orbits}
\par Suppose $x\in X_{k,n}$ is a semisimple matrix (for some $k<n-1$). Our first mission is to decompose the conjugation actions of $H_{k,n}^x$ and its subgroup $P_{k,n}^x$ on $\mathbf{L}_{k,n}^{Ad(x)}(F)$ into a product of simpler action spaces. Regarding the group $H_{k,n}^x$, such a decomposition can surely be deduced from similar descriptions in \cite{jac-ral} and \cite{aiz-gur1}, but we prefer to reproduce it here directly for completeness.

\begin{prop}\label{structure}
\textbf{1.} The action space $\left(H_{k,n}^x,\mathbf{L}_{k,n}^{Ad(x)}(F)\right)$ decomposes as a product of the action spaces $\left\{(\mathbf{GL}_{l_i}(E_i), \mathbf{M}_{l_i,l_i}(E_i))\right\}_{i=1}^t$ and of $\{(H_{m_i,l_i}, \mathbf{L}_{m_i,l_i}(F))\}_{t<i\leq s}$,
where $0\leq t \leq s\leq t+2$, for every $1\leq i\leq t$, $F<E_i$ is a finite field extension and $l_i$ is a positive integer; and for every $t< i\leq s$, $m_i\leq k, l_i\leq n$ are positive integers. All actions in the decomposition are by conjugation.\\ \\
\textbf{2.} The subgroup $P^x_{k,n}$ of $H^x_{k,n}$ is admissible to the decomposition above. Therefore, the action space $\left(P_{k,n}^x,\mathbf{L}_{k,n}^{Ad(x)}(F)\right)$ is a product of $\left\{(P^i, \mathbf{M}_{l_i,l_i}(E_i)\right\}_{i=1}^t$ and of $\{(P^i, \mathbf{L}_{m_i,l_i}(F)\}_{t<i\leq s}$, for certain subgroups $P_i$. Those are given as follows:
$$P^i = \left\{\begin{array}{ccc} \mathbf{GL}_{l_i}(E_i) & \alpha_i=0 & 1\leq i\leq t \\ \mathbf{P}_{l_i}(E_i) & \alpha_i =1 \\ H_{m_i,l_i} & \alpha_i=0 & t<i\leq s \\ P_{m_i,l_i} & \alpha_i=1   \end{array} \right.$$

where $\alpha(x) = (\alpha_1,\ldots, \alpha_s)\in \{0,1\}^s$ is a fixed tuple.  \\ \\
\textbf{3.} After identifying $P^x_{k,n}$ with $\prod_{i=1}^s P^i$, the relative modular character of $P^x_{k,n}$ inside $P_{k.n}$ is given by
$$\Delta_{P_{k,n}/P^x_{k,n}}(g_1,\ldots,g_s) = \prod_{\{i\;:\; \alpha_i=0\}} |\det \widetilde{g_i}|^{-1}\; ,$$
where 
$$\widetilde{g_i} = \left\{ \begin{array}{lll} g_i & 1\leq i\leq t \\ h_i & t<i\leq s & H_{m_i,l_i} =  G_{m_i}\times G_{l_i-m_i},\, g_i= (f_i,h_i) \end{array} \right. .$$

\end{prop}
\begin{proof}
\textbf{1.} Consider the action of $G_n$ by right-multiplication on the row vector space $F^n$. The semisimple operator $x$ gives rise to a decomposition of $F^n= \oplus_j V_j$ to a direct sum of vector spaces, such that $V_j$ has an $E_j = F(\zeta_j)$-vector space structure for a finite field extension $F<E_j$, and $x|_{V_j}$ is acting as multiplication by $\zeta_j$. Grouping isomorphic extensions, we can assume the actions of the $\zeta_{j}$'s are not isomorphic for distinct $j$'s. Then it is easy to see that such a decomposition of $F^n$ gives rise to a decomposition of the conjugation action space $\left(G_n^x,\mathbf{M}_{n,n}(F)^{Ad(x)}\right)$ into action spaces of the form $(\mathbf{GL}_{l_j}(E_j), \mathbf{M}_{l_j,l_j}(E_j))$, with $l_j\leq n$. It remains to see what happens when we add the requirement for commutation relations with $\epsilon_{k,n}$.
\par Since $x \in X_{k,n}$, we have $\epsilon_{k,n}x\epsilon_{k,n} =\theta_{k,n}(x) = x^{-1}$. Thus, the action of $\epsilon_{k,n}$ on $F^n$ must permute the $V_j$'s. Since $\epsilon_{k,n}$ is an involution, we conclude that $\left(H_{k,n}^x,\mathbf{L}_{k,n}^{Ad(x)}(F)\right)$ decomposes as a product of spaces either of the form 
$$\left((\mathbf{GL}_{l_{j_1}}(E_{j_1})\times \mathbf{GL}_{l_{j_2}}(E_{j_2}))^{\epsilon_{k,n}}, (\mathbf{M}_{l_{j_1},l_{j_1}}(E_{j_1})\times \mathbf{M}_{l_{j_2},l_{j_2}}(E_{j_2}))^{\widehat{Ad(\epsilon_{k,n})}}\right)$$
where $j_1\neq j_2$ are such that $\epsilon_{k,n}(V_{j_1}) = V_{j_2}$, or of the form
$$\left(\mathbf{GL}_{l_{j}}(E_{j})^{\epsilon_{k,n}}, \mathbf{M}_{l_{j},l_{j}}(E_{j})^{\widehat{Ad(\epsilon_{k,n})}}\right)$$
where $V_j$ is $\epsilon_{k,n}$-invariant.
Let us show that all of these are isomorphic to the action spaces described in the statement.
\par In the former case (\textit{case (1)}), the acting group is clearly given by $\{(g, \epsilon_{k,n}g\epsilon_{k,n}|_{V_{j_2}})\,:\, g\in \mathbf{GL}_{l_{j_1}}(E_{j_1})\}$, while the space is $\{(A, -\epsilon_{k,n}A\epsilon_{k,n}|_{V_{j_2}})\,:\, A\in \mathbf{M}_{l_{j_1},l_{j_1}}(E_{j_1})\}$ . Hence, the situation is isomorphic to $\mathbf{GL}_{l_{j_1}}(E_{j_1})$ acting on $\mathbf{M}_{l_{j_1},l_{j_1}}(E_{j_1})$.
\par The latter case should itself be separated into two cases. First, assume $x^{-1}|_{V_j} = \theta_{k,n}(x)|_{V_j} \neq x|_{V_j}$ (\textit{case (2)}). Then, on $V_j$, $\theta_{k,n}$ serves as a non-trivial $F$-linear involution of $E_j$. Let $F<L_j<E_j$ be its fixed sub-field ($[E_j:L_j]=2$). This means we can write $V_j\cong E_j\otimes_{L_j} W_j$ for an $L_j$-subspace $W_j\subset V_j$, with $\epsilon_{k,n}$ acting as the non-trivial Galois automorphism in $Gal(E_j/L_j)$ on the $L_j$ component. In these terms, we have
$$ \mathbf{M}_{l_{j},l_{j}}(E_{j})^{\widehat{Ad(\epsilon_{k,n})}} \cong (E_j\otimes_{L_j}\mathbf{M}_{l_{j},l_{j}}(L_{j}))^{\widehat{Ad(\epsilon_{k,n})}} = \tau\otimes_{L_j}\mathbf{M}_{l_{j},l_{j}}(L_{j})$$
where $\tau\in L_i$ is an element for which $\epsilon_{k,n}(\tau\otimes w) = -\tau\otimes w$. Similarly, we have $\mathbf{GL}_{l_{j}}(E_{j})^{\epsilon_{k,n}}\cong \mathbf{GL}_{l_{j}}(L_{j})$, and the action space is isomorphic to $(\mathbf{GL}_{l_j}(L_j), \mathbf{M}_{l_j,l_j}(L_j))$.
\par Finally, still in the case of $\left(\mathbf{GL}_{l_{j}}(E_{j})^{\epsilon_{k,n}}, \mathbf{M}_{l_{j},l_{j}}(E_{j})^{\widehat{Ad(\epsilon_{k,n})}}\right)$, if $x|_{V_j}= x^{-1}|_{V_j}$ (\textit{case (3)}), then $x=\pm I_{V_j}$. That means at once that $E_j = F$ and that there are at most two of such factors in the product. Furthermore, we see that under a suitable change of basis, $\epsilon_{k,n}|_{V_j}$ equals $\epsilon_{m_j,l_j}$. Since $m_j$ is the dimension of the eigenspace of $\epsilon_{k,n}|_{V_j}$ corresponding to the eigenvalue $1$, it obviously cannot exceed $k$.\\

\par \textbf{2.} Let $\{v_j\in V_j\}$ be the decomposition of the vector $e=(0,\ldots,0,1)\in F^n = \oplus_j V_j$. Since the subspaces $V_j$ are invariant under the $H^x_{k,n}$-action on $F^n$ and since $P^x_{k,n}$ is the stabilizer of $e$, the subgroup can also be described as the intersection of the stabilizers of the vectors $v_j$ inside $H^x_{k,n}$. This clearly means that $P^x_{k,n}$ is an admissible subgroup for the above product decomposition. It is left to recognize the stabilizer subgroups of $v_j$ in each of the three cases appearing in the product. 
\par In case (1), since $\epsilon_{k,n}(e) = -e$, we see that $v_{j_1} = -v_{j_2}$. If $v_{j_1}=0$, the stabilizer is clearly the whole acting group, and we set $\alpha_i= 0$ for the $i$ with which we are dealing. Otherwise, we set $\alpha_i=1$ and the subgroup of $\mathbf{GL}_{l_{j_1}}(E_{j_1})$ which stabilizes $(v_{j_1},-v_{j_2})$ becomes $\mathbf{P}_{l_{j_1}}(E_{j_1})$, up to a change of basis.
\par In case $V_j$ is $\epsilon_{k,n}$-invariant, again we set $\alpha_i=0$ for the corresponding $i$ when $v_j=0$, and the stabilizer is then the whole acting group. Otherwise we set $\alpha_i=1$. In case (2), the condition $\epsilon_{k,n}(e)=-e$ imposes $v_j = \tau\otimes w$, whose stabilizer is again just a stabilizer of one non-zero vector in $\mathbf{GL}_{l_j}(L_{j})$, i.e.~isomorphic to $\mathbf{P}_{l_j}(L_j)$. 
\par As for the last case, it is easily seen that given a vector $v_j$ for which $\epsilon_{k,n}|_{V_j}(v_j) = -v_j$, a suitable change of basis exists such that $H_{m_j,l_j}$ remains at place while the stabilizer of $v_j$ becomes $P_{m_j,l_j}$.\\

\par \textbf{3.} Since $\mathbf{GL}_l(E)$ is a unimodular group and the modular character of $\mathbf{P}_l(E)$ is $|\det^{-1}|$, it is easy to see that $\Delta_{P^x_{k,n}}(g)= \prod_{i\;:\;\alpha_i=1} |\det(\widetilde{g_i})|^{-1}$, for all $g=(g_1,\ldots,g_s)\in \prod_{i=1}^s P^i\cong P^x_{k,n}$. Now, $\Delta_{P_{k,n}}^{-1}$ is given by the norm of the determinant of the lower block, that is, the determinant of the restriction of the operator to the $-1$ eigenspace of $\epsilon_{k,n}$. We need to compute it for each $P^i$. 
\par In case (1), we are interested in the restriction of $P^i$ to the space $\{(v,-\epsilon_{k,n}(v))\}_{v\in V_{j_1}}$, which is isomorphic to $\mathbf{GL}_{l_{j_1}}(E_{j_1})$. In case (2), we are looking on its action on the space $\tau\otimes W_j$ which gives the same conclusion. Finally, for case (3), the restriction is the projection on the right component of $H_{m_j,l_j}= G_{m_j}\times G_{l_j-m_j}$. 
\par Thus, in all cases we have $\Delta_{P_{k,n}}(g_i) = |\det( \widetilde{g_i})|^{-1}$, and the statement easily follows.

\end{proof}

The above decomposition can also be applied for the study of the geometry of the $P_{k,n}$-action on $X_{k,n}$ in the following way.

\begin{prop}\label{fin-orbit}
Let $\mathcal{O}$ be a closed $H_{k,n}$-orbit in $X_{k,n}$. Then, $\mathcal{O}$ splits into a finite number of $P_{k,n}$-orbits, with a unique open (in $\mathcal{O}$) orbit $A\subset \mathcal{O}$. Furthermore, $A$ consists exactly of those $x\in \mathcal{O}$ for which $\alpha(x) = (1,\ldots,1)$.
\end{prop}
\begin{proof}
Since $H^x_{k,n}$ is the stabilizer of $x\in \mathcal{O}$ in $H_{k,n}$, we can fix one $x\in \mathcal{O}$ and identify $\mathcal{O}\cong H_{k,n}/H^x_{k,n}$. Note, that instead of counting orbits of $P_{n,k}$ on $H_{n,k}/H^x_{k,n}$, we can count the orbits of the right action of $H^x_{k,n}$ on the space $P_{n,k}\backslash H_{k,n}$. Consider $F^{n-k}\setminus \{0\}$ as a row space on which $H_{k,n}$ acts transitively by right matrix multiplication with its lower block. Then, $P_{k,n}$ will be the stabilizer of $e=(0,\ldots,0,1)$, and $F^{n-k}\setminus \{0\}\cong P_{n,k}\backslash H_{n,k}$. 
\par By Proposition \ref{closed}, $x$ is a semisimple matrix, hence, we can apply the previous proposition. Since $F^{n-k}$ lies inside $F^{n}$ as the eigenspace of $\epsilon_{k,n}$ related to the value $-1$, we can use the same reasoning as in the previous proof to identify it as the set of vectors of the form $((w_i,-\epsilon_{k,n}(w_i)),(\tau\otimes w_i), w_i)$, with the notation that corresponds to the 3 cases classification of the previous proof. We are left to count the non-zero orbits of $H^x_{k,n}$ on this subspace. 
\par Clearly, these are determined by the nullity of each of $\{w_i\}_{i=1}^s$, and we see that there are exactly $2^s-1$ such orbits. In particular, there is only one open orbit among them, namely, the set of vectors for which all $w_i\neq 0$. We denote by $A$ the corresponding unique open orbit in $\mathcal{O}$.
\par The $P_{k,n}$-orbit of $x$ inside $\mathcal{O}$ corresponds to the $H^x_{k,n}$-orbit of the identity in $P_{n,k}\backslash H_{n,k}$, or in other words, of the vector $e$ in $F^{n-k}$. Thus, the question of whether $x$ belongs to $A$ is equivalent to asking whether the components of $e$ in the $\{w_i\}$ (or $\oplus_j V_j$) decomposition are all non-zero. This is equivalent to the condition $\alpha(x)=(1,\ldots,1)$ from the way we defined $\alpha(x)$.
\end{proof}

\subsection{Distributions on $X_{k,n}$}
\begin{cor}\label{first-cor}
If $A$ is the open $P_{k,n}$-orbit inside a closed $H_{k,n}$-orbit $\mathcal{O}\subset X_{k,n}$ and $x\not\in A$, then $\mathcal{D}(N_x\mathcal{O})^{P^x_{k,n},\Delta_{P_{k,n}/P^x_{k,n}}^{-1}}$ is trivial.
\end{cor}
\begin{proof}
First by Proposition \ref{closed}, we may prove that $\mathcal{D}(\mathbf{L}_{k,n}^{Ad(x)}(F))^{P^x_{k,n},\Delta_{P_{k,n}/P^x_{k,n}}^{-1}}$ is trivial.  Since $x\not\in A$, it must have some index $\alpha_i=0$. As observed in Proposition \ref{structure}, it follows there must be a subgroup $M$ inside $P^x_{k,n}$ isomorphic to either $\mathbf{GL}_{l_i}(E)$ or $H_{m_i,l_i}$ whose action on $\mathbf{L}_{k,n}^{Ad(x)}(F)$ is isomorphic to a conjugation action on a specified matrix space. Note also, that by Proposition \ref{structure}.3, $\Delta_{P_{k,n}/P_x}(g) = |\det(\tilde{g})|$ for $g\in M$, with $\tilde{g}$ defined as in the mention proposition. Take $g= \omega I\in M$ for some $\omega \in F$ with $|\omega|>1$. Then it is clear that $\Delta_{P_{k,n}/P_x}(g)\not=1$, while $g$, being a scalar operator, must act trivially on $\mathbf{L}_{k,n}^{Ad(x)}(F)$. This shows there cannot be any non-trivial $P^x_{k,n},\Delta_{P_{k,n}/P_x}^{-1}$-invariant distributions on that space.
\end{proof}

We are now ready to claim the first distributional reduction of the main problem.

\begin{lem}\label{cruc}
Let $k,n$ be positive integers such that $k <n-1$. Suppose that for every semisimple matrix $x\in X_{k,n}$ the equality
$$\mathcal{D}\left(\mathbf{L}_{k,n}^{Ad(x)}(F)\right)^{P^x_{k,n}} = \mathcal{D}\left(\mathbf{L}_{k,n}^{Ad(x)}(F)\right)^{H^x_{k,n}}$$
holds, where the group action is by conjugation. Then, we also have
$$\mathcal{D}(X_{k,n})^{P_{k,n}}= \mathcal{D}(X_{k,n})^{H_{k,n}}$$
\end{lem}
\begin{proof}
Let $\mathcal{O}$ be a closed $H_{k,n}$-orbit in $X_{k,n}$. Since $X_{k,n}\cong G_n/H_{k,n}$, we can consider it as an open and closed subset of $Y:=\mathbf{GL}_n/\mathbf{GL}_k\times \mathbf{GL}_{n-k}(F)$. By the Luna's slice theorem (Theorem \ref{luna}), there exists an open $H_{k,n}$-invariant subset $\mathcal{O}\subset U_{\mathcal{O}}\subset Y$, together with an $H_{k,n}$-equivariant retract $p:U_{\mathcal{O}}\to \mathcal{O}$ possessing certain properties. By Proposition \ref{covered}, the union of all such $U_\mathcal{O}$ over all closed orbits $\mathcal{O}$, together with $Y\setminus X_{k,n}$, cover $Y$. Hence, $X_{k,n}\subset \bigcup_{\mathcal{O}} U_\mathcal{O}$. 
\par If we show $\mathcal{D}(U_\mathcal{O})^{P_{k,n}} = \mathcal{D}(U_{\mathcal{O}})^{H_{k,n}}$ for every closed $H_{k,n}$-orbit $\mathcal{O}$, we will know that $\mathcal{D}(\bigcup_{\mathcal{O}} U_\mathcal{O})^{P_{k,n}} = \mathcal{D}(\bigcup_{\mathcal{O}} U_{\mathcal{O}})^{H_{k,n}}$, for example, by \cite[Lemma 3.2]{off}. That will prove the statement, since $X_{k,n}$ is closed and open in the union.
\par Let $A\subset \mathcal{O}$ be the open $P_{k,n}$-orbit which was shown to exist in Proposition \ref{fin-orbit}, and choose $x\in A$. As usual $\mathcal{O}\cong H_{k,n}/H^x_{k,n}$, and $x$ is semisimple by Proposition \ref{closed}. The result will be achieved by applying Proposition \ref{abstract-lem} on the retract $p$ and the groups $P_{k,n}<H_{k,n}$. Indeed, condition (ii) holds because $H_{k,n}$ and $H^x_{k,n}$ are unimodular (this is a consequence of Proposition \ref{structure}.1), and (iii) holds because of Proposition \ref{fin-orbit}. We are left to show that conditions (iv) and (v) hold for $x$.
\par Recall that Luna's slice comes with an $H^x_{k,n}$-equivariant embedding $\iota:p^{-1}(x)\to N_x$ whose image is open and saturated. Hence, there is an $H^x_{k,n}$-invariant map $\pi:N_x\mathcal{O}\to F^m$, and an open $V\subset F^m$ such that $i(p^{-1}(x)) = \pi^{-1}(V)$. Consider $\xi\in \mathcal{D}(p^{-1}(x))^{P^x_{k,n}}$, and suppose there is an $h\in H^x_{k,n}$ with $\xi\not=\xi^h$. Pick $t\in Supp(\xi-\xi^h)$ and an open compact neighborhood $B$ of $\pi(\iota(t))$ in $V$. Then, $\pi^{-1}(B)$ is an open and closed subset of $N_x\mathcal{O}$ that contains $\iota(t)$ and contained in $\iota(p^{-1}(x))$. Now, we have a well-defined operator $\alpha_B:\mathcal{D}(p^{-1}(x))\to \mathcal{D}(N_x\mathcal{O})$ given by $\alpha_B(\eta)(f) = \eta((f\circ \iota)|_{\pi^{-1}(B)})$ for $f\in \mathcal{S}(N_x\mathcal{O})$. Hence, $\alpha_B(\xi)\in \mathcal{D}(N_x\mathcal{O})^{P^x_{k,n}}$, but by Proposition \ref{closed} and our assumption this is the same as $\alpha_B(\xi)\in \mathcal{D}(N_x\mathcal{O})^{H^x_{k,n}}$. Therefore, $\alpha_B(\xi^h) = \alpha_B(\xi)^h = \alpha_B(\xi)$, which gives $\alpha_B(\xi-\xi^h)=0$. This contradicts the clear fact that $\iota(t)\in Supp(\alpha_B(\xi-\xi^h))$. Finally, $\mathcal{D}(p^{-1}(x))^{P^x_{k,n}}= \mathcal{D}(p^{-1}(x))^{H^x_{k,n}}$ holds and condition (iv) of Proposition \ref{abstract-lem} is satisfied.
\par With similar arguments, we can deduce condition (v) from Corollary \ref{first-cor} by noting that $p^{-1}(y)$ can be embedded with an open saturated image inside $N_y\mathcal{O}$ for all $y\in \mathcal{O}$, and directly using \cite[Lemma 3.1.3]{aiz-gur1}.
\end{proof}

\begin{thm}\label{main-reduction}
Let $k,n$ be positive integers such that $k <n-1$. Suppose that the equality
$$\mathcal{D}\left(\mathbf{L}_{m,l}(F)\right)^{P_{m,l}} = \mathcal{D}\left(\mathbf{L}_{m,l}(F)\right)^{H_{m,l}}$$
holds, for all positive integers $m\leq k$ and $l\leq n$. Then, we also have 
$$\mathcal{D}(X_{k,n})^{P_{k,n}}= \mathcal{D}(X_{k,n})^{H_{k,n}}$$
\end{thm}
\begin{proof}
Suppose $\{G^i\}$ are finitely many lctd groups acting on lctd spaces $\{X^i\}$ respectively, and $H^i<G^i$ are fixed subgroups. Let $(H,X)$ be the product of all $\{(H^i,X^i)\}$'s, and $(G,X)$ the product of $\{(G^i,X^i)\}$'s. If $\mathcal{D}(X^i)^{P^i} = \mathcal{D}(X^i)^{H^i}$ for all $i$, then it is easy to check that $\mathcal{D}(X)^H = \mathcal{D}(X)^G$. See, for example \cite[3.1.5]{aiz-gur3}.
\par So, combining that reasoning with Proposition \ref{structure} and Lemma \ref{cruc}, it is enough to show that $\mathcal{D}(\mathbf{M}_{l,l}(E))^{\mathbf{P}_l(E)} = \mathcal{D}(\mathbf{M}_{l,l}(E))^{\mathbf{GL}_l(E)}$ and that $\mathcal{D}\left(\mathbf{L}_{m,l}(F)\right)^{P_{m,l}} = \mathcal{D}\left(\mathbf{L}_{m,l}(F)\right)^{H_{m,l}}$, for certain finite field extensions $F<E$ and integers $m\leq k$, $l\leq n$.
\par The former equality was proved by Bernstein in \cite{bern}, while the latter is the assumption, when $m>0$. When $m=0$ there is, of course, nothing to prove.
\end{proof}

\subsection{Reduction to nilpotents}
Let $\mathcal{N}_{k,n}\subset \mathbf{L}_{k,n}(F)$ be the cone of nilpotent matrices. 
\par Before moving to tackle the equality of distribution spaces on the linear space, we would like to reduce our problem further to that of distribution spaces on $\mathcal{N}_{k,n}$. For that cause, we apply some of the ideas which appeared in  \cite[Theorem 3.2.1]{aiz-gur1}.

\begin{lem}\label{closedisclosed}
The set $\mathcal{N}_{k,n}$ is closed in $\mathbf{L}_{k,n}(F)$, and every closed $H_{k,n}$-orbit in $\mathbf{L}_{k,n}(F)\setminus\mathcal{N}_{k,n}$ remains closed in $\mathbf{L}_{k,n}(F)$.
\end{lem}
\begin{proof}
Let $t:\mathbf{L}_{k,n}(F)\to F^{n}$ be the map given by the coefficients of the characteristic polynomials of elements of $\mathbf{L}_{k,n}(F)$ (as matrices). Clearly, $t$ is $H_{k,n}$-invariant, and $\mathcal{N}_{k,n}= t^{-1}(1,0,\ldots,0)$. Every $H_{k,n}$-orbit is contained in a single fiber of $t$, hence, its closure remains in that fiber.
\end{proof}

\begin{prop}\label{redtonil}
Let $k,n$ be positive integers such that $k <n-1$. Suppose that the equality
$$\mathcal{D}\left(\mathbf{L}_{m,l}(F)\right)^{P_{m,l}} = \mathcal{D}\left(\mathbf{L}_{m,l}(F)\right)^{H_{m,l}}$$
holds, for all positive integers $m\leq k$ and $l< n$. Then,
$$\mathcal{D}(\mathbf{L}_{k,n}(F)\setminus \mathcal{N}_{k,n})^{P_{k,n}}=\mathcal{D}(\mathbf{L}_{k,n}(F)\setminus \mathcal{N}_{k,n})^{H_{k,n}}.$$
\end{prop}

\begin{proof}
Let $\mathcal{O} \subset \mathbf{L}_{k,n}(F)\setminus \mathcal{N}_{k,n}$ be a closed $H_{k,n}$-orbit. As in the proof of Lemma \ref{cruc}, it is enough to find for each such $\mathcal{O}$ an open $H_{k,n}$-invariant neighborhood $\mathcal{U}$, such that $\mathcal{D}(\mathcal{U})^{P_{k,n}}=\mathcal{D}(\mathcal{U})^{H_{k,n}}$.
\par Consider the identity element $I_n=\rho_{k,n}(I_n)\in X_{k,n}$. The action of $H_{k,n}$ on its normal space is equivalent to that of conjugation on $\mathbf{L}_{k,n}(F)^{Ad(I_n)} = \mathbf{L}_{k,n}(F)$. By applying Luna's Slice Theorem on the trivial closed orbit of $I_n$ in $X_{k,n}$, we get an open $H_{k,n}$-invariant neighborhood $I_n\in U\subset X_{k,n}$ with an $H_{k,n}$-equivariant embedding $\iota:U\to \mathbf{L}_{k,n}(F)$, whose image is open and contains the zero vector.
\par Note, that $\lambda \mathcal{O}\subset \iota(U)$ for some $\lambda\in F^\times $, and $\lambda \mathcal{O}$ is also a closed orbit. Since the action on $\mathbf{L}_{k,n}(F)$ is linear, finding a suitable neighborhood $\mathcal{U}$ for it, would give $\lambda^{-1}\mathcal{U}$ as our desired neighborhood for $\mathcal{O}$. Thus, we can safely assume that $\mathcal{O}\subset\iota(U)$.
\par Another application of Luna's Slice Theorem, this time on the space $\mathbf{L}_{k,n}(F)\setminus \mathcal{N}_{k,n}$, would give a prescribed open neighborhood $\mathcal{U}$ of $\mathcal{O}$. Now, since $\iota$ is an equivariant embedding and since by Lemma \ref{closedisclosed} $\mathcal{O}$ is closed in $\iota(U)$ as well, the normal spaces of elements in $\mathcal{O}$ together with the actions of the stabilizers in $H_{k,n}$ are isomorphic to the normal spaces of elements in $\iota^{-1}(\mathcal{O})$. Thus, using the same arguments as the proof of Lemma \ref{cruc} it can be shown that the equality $$\star \quad \mathcal{D}\left(\mathbf{L}_{k,n}^{Ad(x)}(F)\right)^{P^x_{k,n}} = \mathcal{D}\left(\mathbf{L}_{k,n}^{Ad(x)}(F)\right)^{H^x_{k,n}}$$
for a certain $x\in U$ with $\iota(x)\in \mathcal{O}$, would imply $\mathcal{D}(\mathcal{U})^{P_{k,n}}=\mathcal{D}(\mathcal{U})^{H_{k,n}}$.
\par Indeed, since $\iota(x)\not=0$, we clearly have $H^x_{k,n} = H^{\iota(x)}_{k,n}\subsetneq H_{k,n}$. Hence, 
$$\dim \mathbf{L}_{k,n}^{Ad(x)}(F) = \dim N_{x}\iota^{-1}(\mathcal{O}) = \dim X_{k,n} - \,\dim H_{k,n}/H^x_{k,n}< \dim \mathbf{L}_{k,n}(F).$$ 
We see that the factor $\mathbf{L}_{k,n}(F)$ cannot appear in the decomposition of Proposition \ref{structure} for $\mathbf{L}_{k,n}^{Ad(x)}(F)$. Therefore, using the same technique as in the proof of Theorem \ref{main-reduction}, our assumptions are enough to prove the equality $\star$.

\end{proof}

\begin{cor}\label{nil-cor}
Suppose that the equality
$$\mathcal{D}\left(\mathbf{L}_{m,l}(F)\right)^{P_{m,l}} = \mathcal{D}\left(\mathbf{L}_{m,l}(F)\right)^{H_{m,l}}$$
holds, for all positive integers $m\leq k$ and $l< n$, and suppose further that $\mathcal{D}(\mathcal{N}_{k,n})^{P_{k,n}}= \mathcal{D}(\mathcal{N}_{k,n})^{H_{k,n}}$. Then, $\mathcal{D}\left(\mathbf{L}_{k,n}(F)\right)^{P_{k,n}} = \mathcal{D}\left(\mathbf{L}_{k,n}(F)\right)^{H_{k,n}}$ holds.
\end{cor}
\begin{proof}
We use the localization principle from \cite[1.4]{bern} with the map $t$ from the proof of Lemma \ref{closedisclosed}. By that principle, it is enough to prove the equality of invariant distribution spaces only on the fibers of $t$. Each of fibers, except $t^{-1}(1,0,\ldots,0)=\mathcal{N}_{k,n}$, is closed in $L_{k,n}(F)\setminus\mathcal{N}_{k,n}$. Hence, the equality of distribution spaces on each of those can be deduced from the previous proposition. 
\end{proof}

\begin{thm}\label{final-form}
Let $k,n$ be positive integers such that $k <n-1$. Suppose that the equality
$$\mathcal{D}\left(\mathcal{N}_{m,l}(F)\right)^{P_{m,l}} = \mathcal{D}\left(\mathcal{N}_{m,l}(F)\right)^{H_{m,l}}$$
holds, for all positive integers $m\leq k$ and $l\leq n$. Then, we also have 
$$\mathcal{D}(X_{k,n})^{P_{k,n}}= \mathcal{D}(X_{k,n})^{H_{k,n}}$$
\end{thm}
\begin{proof}
By induction on $m$ and $l$, it is easy to use Corollary \ref{nil-cor} at each step in such manner that the assumptions of Theorem \ref{main-reduction} are satisfied.
\end{proof}

\section{The case $k=1$.}
In this section we fix $n\geq3$. We would like to characterize the space of $P_{1,n}$-conjugation-invariant distributions on the nilpotent cone $\mathcal{N}_{1,n}$. The eventual result will be that it is a $2$-dimensional space spanned by the obvious point distribution centered on the zero vector, and by a second distribution $\nu$ which we will need to construct, both of which are also $H_{1,n}$-invariant. To show the existence of such $\nu$ we need to be able to prolong an $H_{1,n}$-invariant distribution on an open orbit onto its closure in such way that it remains invariant. This task of singularity resolution is carried out using the "dense Frobenius descent" (Proposition \ref{dense-frob}) technique, which reduces the problem to that of distribution prolongation from $F^\times$ to $F$.
\par The space $\mathbf{L}_{1,n}(F)$ is the set of matrices given by
$$\left\{\left(\begin{array}{ll} 0 & \overline{v} \\ \overline{w} & 0 \end{array}\right): \overline{v}\in \mathbf{M}_{1,n-1} (F),\;\overline{w}\in \mathbf{M}_{n-1,1}(F)\right\}.$$
Thus we can parametrize this space by pairs $(\overline{v},\overline{w})$, where the first is a row a vector while the second is a column vector. It will sometimes be convenient to write the row vector as $\overline{v} = (\overline{v}'\;a)$ and the column vector as $\overline{w} = \left(\begin{array}{ll} \overline{w}' \\ b\end{array}\right)$, where $a,b\in F$ and $\overline{v}',\overline{w}'$ are of length $n-2$. In this notation we get a parametrization of $\mathbf{L}_{1,n}(F)$ by quadruples $(\overline{v}',a, \overline{w}',b)$.
\par The conjugation action of the group $H_{1,n}$, naturally viewed as $F^\times\times G_{n-1}$, is given in pairs notation by the formula $(\alpha, A)\cdot (\overline{v},\overline{w}) = (\alpha \overline{v}A^{-1}, \alpha^{-1}A\overline{w})$. Also, it can be seen that in these terms the nilpotent cone $\mathcal{N}_{1,n}$ consists of matrices for which $\overline{v}\cdot \overline{w}=0$. We denote by $U\subset \mathcal{N}_{1,n}$ the open $H_{1,n}$-orbit defined by the condition $\overline{v},\overline{w}\neq 0$.

\begin{lem}\label{thereis-dist-gen}
On $\mathcal{N}_{1,n}$ there is an $H_{1,n}$-invariant distribution, whose support is the whole space. 
\end{lem}
\begin{proof}
Consider the compact projective space $\mathbb{P}^{n-2}(F)$ together with the $H_{1,n}$-action given by $(\alpha, A)\cdot [p] = [Ap]$ for any line $[p]$ represented by $p\in F^{n-1}$ and $(\alpha,A)\in F^\times\times G_{n-1}$. We define an $H_{1,n}$-equivariant map $i:U\to \mathbb{P}^1(F)$ by $i((\overline{v},\overline{w}))=[\overline{w}]$.
\par Note, that the closure of $i^{-1}\left(\left[\left( 0  \cdots 0 \;1\right)^t\right]\right)$ is 
$$Y:= \left\{ (\overline{v}',0, \overline{0}, b)\,:\, \overline{v}'\in \mathbf{M}_{1,n-2}(F),\, b\in F \right\}\cong F^{n-2}\times F.$$ 
Note also, that 
$$S:= Stab\left(H_{1,n}, \left[\left( 0  \cdots 0 \;1\right)^t\right]\in\mathbb{P}^{n-2}(F)\right) =$$ $$= F^\times\times \left\{ \left(\begin{array}{ll} B & 0\\ \overline{\gamma} &d \end{array} \right)\in G_{n-1}(F) \,:\, B\in G_{n-2}(F),\,d\in F^\times\right\}.$$
Since $U$ is dense in $\mathcal{N}_{1,n}$, by Proposition \ref{dense-frob}, it is enough to exhibit a non-zero $S,\Delta_{H_{1,n}/S}^{-1}$-invariant distribution supported on $Y$. The action of $S$ on $Y$ is given by
$$\left(\alpha, \left(\begin{array}{ll} B & 0\\ \overline{\gamma} &d \end{array} \right)\right)\cdot (\overline{v}',0, \overline{0}, b)= (\alpha \overline{v}'B^{-1},0, \overline{0}, \alpha^{-1}d b).$$
Also, we have
$$\Delta_{H_{1,n}(F)/S}^{-1}\left(\alpha, \left(\begin{array}{ll} B & 0\\ \overline{\gamma} &d \end{array} \right)\right) = \Delta_{S}\left(\alpha,\left(\begin{array}{ll} B & 0\\ \overline{\gamma} &d \end{array} \right)\right) = |\det(B)d^{2-n}|.$$
The last equality is easily derived after noting that $S$ is a direct sum of the unimodular $F^\times$, and of a parabolic subgroup of $G_{n-1}$. 
\par Let $\chi(x) = |x|^{2-n}$ be a character on $F^\times$, and $0\neq\mu\in \mathcal{D}(F^{\times})^{F^\times, \chi}$ the corresponding distribution. It is known that such $\mu$ can be prolonged into a non-zero distribution $\widetilde{\mu}\in \mathcal{D}(F)^{F^\times, \chi}$ supported on all of $F$ (see \cite[Chapter 2, 2.3]{gel-gp}, or \cite{weil66}). After identifying $Y$ with $F^{n-2}\times F$, we can put the distribution $m\otimes \widetilde{\mu}$ on it, where $m$ is the Haar measure on the vector space $F^{n-2}$. It is easy to see that the $S$-action transforms this distribution according to the formula for $\Delta_{H_{1,n}(F)/S}^{-1}$. Thus, we have our desired distribution.
\end{proof}

Let $\nu\in\mathcal{D}(\mathcal{N}_{1,n})^{H_{1,n}}$ be a distribution as provided by the lemma. Denote also by $\delta_0\in\mathcal{D}(\mathcal{N}_{1,n})^{H_{1,n}}$ the distribution given by $\delta_0(f) = f(\overline{0},\overline{0})$, for all $f\in\mathcal{S}(\mathcal{N}_{1,n})$.

\begin{thm}\label{k1-thm}
The space of $P_{1,n}$-invariant distributions on $\mathcal{N}_{1,n}$ is spanned by $\delta_0$ and $\nu$. In particular,
$$\mathcal{D}(\mathcal{N}_{1,n})^{P_{1,n}}=\mathcal{D}(\mathcal{N}_{1,n})^{H_{1,n}}$$

\end{thm}
\begin{proof}
Let $T\in \mathcal{D}(\mathcal{N}_{1,n})^{P_{1,n}}$ be given.
\par Note that $W=\left\{(\overline{v}',a, \overline{w}',b)\in U\;:\; b\neq 0\right\}$ is an open $P_{1,n}$-orbit in $\mathcal{N}_{1,n}$. Since an orbit can have at most one invariant distribution up to a constant and since $\nu|_W\neq 0$, there must be a constant $c$ for which $T-c\nu|_{W} = 0$. Yet, $T-c\nu$ is still $P_{1,n}$-invariant, which means we can simply assume that $T$ has its support inside $\mathcal{N}_{1,n}\setminus W$.
\par In this set, $O_1= \{(\overline{0},0, \overline{w}',b)\,:\, b\neq0\}$ is an open $P_{1,n}$-orbit lacking a non-zero invariant distribution. Indeed, for $s_1= (\overline{0},0, \overline{0},1)\in O_1$, its stabilizer inside $P_{1,n}$ is the unimodular group $$\{1\}\times\left\{\left(\begin{array}{ll} A \\ &1\end{array}\right)\;:\;A\in G_{n-2}\right\}.$$ It can be seen that $\Delta_{P_{1,n}/Stab(P_{1,n},s_1)}\neq 1$. Thus, $supp\, T\subset Y=\left\{(\overline{v}',a, \overline{w}',0)\right\}$.\\
Note, that we naturally have $Y\cong \mathcal{N}_{1,n-1}\times F$ as $P_{1,n}$-spaces, with the action on the right component given by $(\alpha,A)\cdot a = \alpha a$. Under this decomposition we write $T = T_1\otimes T_2$. Since $\left(\alpha,\left(\begin{array}{ll}\alpha I_{n-2}\\ &1 \end{array}\right)\right)\in P_{1,n}$ acts trivially on the left component of the decomposition, $T_2$ is actually an $F^\times$-invariant distribution on $F$, hence, by a well-known fact $T_2$ is supported on $0$ (see, for example, \cite[0.7]{bern}). In other words, $supp\, T \subset \left\{(\overline{v}',0, \overline{w}',0)\right\}$.
\par Now, there are only two $P_{1,n}$-orbits left that are fully contained inside the above set: The zero orbit, and $O_2 = \left\{(\overline{0}',0, \overline{w}',0)\,:\, \overline{w}\neq 0\right\}$. Thus, $T$ can be viewed as a $G_{n-2}$-invariant distribution on the space $F^{n-2}\cong O_2\cup \{0\}$, which means it must be supported on the zero vector, i.e.~a multiple of $\delta_0$. 
\end{proof}

A combination of the above with Theorem \ref{final-form} immediately gives a result about distributions on a non-linear space.
\begin{cor}\label{final-cor}
The equality
$$\mathcal{D}(X_{1,n})^{P_{1,n}}= \mathcal{D}(X_{1,n})^{H_{1,n}}$$
holds.
\end{cor}

\section{Application to Representation Theory}
Our focus will now turn to smooth (admissible) complex representations of the group $G_n$. We will say that an irreducible such representation $(\pi,V)$ is \textit{$H_{k,n}$-distinguished}, if there is a non-zero linear functional on $V$, which is invariant under the dual $H_{k,n}$-action. By Frobenius reciprocity, it easily follows that $\pi$ is $H_{k,n}$-distinguished, if and only if, it can be embedded as a $G_n$-sub-representation of the space of locally constant functions on $G_n/H_{k,n}\cong X_{k,n}$ (with the $G_n$-action given by shifting of functions). Therefore, it is much expected that results on the space of distributions on $X_{k,n}$ would have implications on $H_{k,n}$-distinguished representations of $G_n$.
\par We are interested in the space of $P_{k,n}$-invariant functionals on smooth irreducible representations of $G_n$. More precisely, given an $H_{k,n}$-distinguished such $(\pi,V)$, we would like to claim that $(V^\ast)^{P_{k,n}} = (V^\ast)^{H_{k,n}}$. 
\par Note, that since the center $Z$ of $G_n$ is contained in $H_{k,n}$, it must act trivially on all $H_{k,n}$-distinguished representations. In case $k=n-1$, we have $H_{k,n}= P_{k,n}Z$, which trivially implies the equality that we seek. The cases $k=0,n$ are also clearly of no interest. Thus, we deal with the condition $1\leq k\leq n-2$.
\par Let us recount a rather standard argument for a deduction of representation-theoretic statements from the results of previous sections. 
\par Let $(\pi,V)$ be a smooth irreducible $H_{k,n}$-distinguished representation of $G_n$. Recall the notion of the smooth contragredient representation $(\widetilde{\pi},\widetilde{V})$. In fact, $\widetilde{\pi}$ can be equivalently realized on the space $V$ with the action $\widetilde{\pi}(g) = \pi(^tg^{-1})$ ($^t$ marks matrix transposition). Since $H_{k,n}=H_{k,n}^t$, the above realization shows that $ (\widetilde{\pi},\widetilde{V})$ is $H_{k,n}$-distinguished as well. Let us fix $0\neq \lambda \in 
\left(\left(\widetilde{V}\right)^\ast\right)^{H_{k,n}}$.
\par A choice of Haar measure $dm$ on $G_n$ defines an action of functions in $\mathcal{S}(G_n)$ on $V$ by $f\cdot v= \int_{G} f(g)\pi(g)v\,dm(g)$, hence, induces a surjective map $A_\pi:\mathcal{S}(G_n)\to End(V)\cong V\otimes \widetilde{V}$. Note, that $G_n\times G_n$ has a left action on $G_n$  with $(g_1,g_2)\cdot g = g_1gg_2^{-1}$, which induces an action on $\mathcal{S}(G_n)$. The mapping $A_\pi$ intertwines that action with the action of $G_n\times G_n$ on $V\otimes \widetilde{V}$.  Dualizing, we have an equivariant embedding $A_\pi^\ast$ of $V^\ast\otimes (\widetilde{V})^\ast$ into $\mathcal{D}(G_n)$. The image of the embedding consists of the so-called \textit{Bessel} distributions of $\pi$. In particular,
$$\iota: V^\ast \to \mathcal{D}(G_n)^{ 1 \times H_{k,n} }, \qquad \iota(\nu) = A^\ast_{\pi}(\nu\otimes \lambda)$$ 
defines an embedding which intertwines the $G_n$-action (the one induced from left multiplication on $G_n$ this time).
\par Finally, it is straightforward to show that $\mathcal{D}(G_n)^{ 1 \times H_{k,n} }$ and $\mathcal{D}(G_n/H_{k,n})$ are isomorphic as $G_n$-spaces (see \cite[3.1]{off}). Going through the identification $G_n/H_{k,n}\cong X_{k,n}$, we see that $V^\ast$ is embedded into $\mathcal{D}(X_{k,n})$ as a $G_n$-space.
\begin{cor}
The equality $\mathcal{D}(X_{k,n})^{P_{k,n}}= \mathcal{D}(X_{k,n})^{H_{k,n}}$ would imply $(V^\ast)^{P_{k,n}} = (V^\ast)^{H_{k,n}}$ for every $H_{k,n}$-distinguished irreducible smooth representation of $G_n$.
\end{cor}
Having shown that deduction, we see that Theorem \ref{final-form} proves Lemma \ref{start-lem}, and Corollary \ref{final-cor} proves Theorem \ref{start-thm}.

\bibliographystyle{abbrv}
\bibliography{propo2}{}

\end{document}